\documentclass[12pt]{amsart}
\usepackage{amssymb,amsmath,amsfonts,   amsthm,nomencl,mathrsfs} 
\usepackage[arrow, matrix, curve]{xy}
\usepackage[latin1]{inputenc}
\usepackage{a4wide}
\newcommand{\mal}{\mathscr{M}}

\newcommand{\IR}{\mathbb{R}}

\newcommand{\question}[1]{\leavevmode{\marginpar{\tiny%
$\hbox to 0mm{\hspace*{-0.5mm}$\leftarrow$\hss}%
\vcenter{\vrule depth 0.1mm height 0.1mm width \the\marginparwidth}%
\hbox to 0mm{\hss$\rightarrow$\hspace*{-0.5mm}}$\\\relax\raggedright #1}}}

\newcommand{\sm}{\sim_b}

\newcommand{\IFF}{\mathscr{F}}

\renewcommand{\c}{ \mathrm{cpt} }

\newcommand{\IL}{L}

\newcommand{\dom}{\mathrm{Dom}}

\newcommand{\IN}{\mathbb{N}}

\newcommand{\IP}{\mathbb{P}}

\newcommand{\Id}{{\rm d}}

\newcommand{\f}{\frac}
\newcommand{\nn}{\nonumber}

\theoremstyle{plain}            % body italics
\newtheorem{theorem}{theorem}[section]

\newtheorem{Corollary}[theorem]{Corollary}
\newtheorem{Theorem}[theorem]{Theorem}

\newtheorem{Propandef}[theorem]{Proposition and definition}

\theoremstyle{definition}       
\newtheorem{Definition}[theorem]{Definition}

\newtheorem{Remark}[theorem]{Remark}
\newtheorem{Example}[theorem]{Example}

\allowdisplaybreaks[1]

\begin{document}

\begin{titlepage} \setcounter{page}{1}
\title[On the geometry of semiclassical limits]{On the geometry of semiclassical limits on Dirichlet spaces}

\author[B. G\"uneysu]{Batu G\"uneysu}

\end{titlepage}

\maketitle 

\begin{abstract} This paper is a contribution to semiclassical analysis for abstract Schrödinger type operators on locally compact spaces: Let $X$ be a metrizable seperable locally compact space, let $\mu$ be a Radon measure on $X$ with a full support. Let $(t,x,y)\mapsto p(t,x,y)$ be a strictly positive pointwise consistent $\mu$-heat kernel, and assume that the generator $H_p\geq 0$ of the corresponding self-adjoint contraction semigroup in $\IL^2(X,\mu)$ induces a regular Dirichlet form. Then, given a function  $\Psi : (0,1)\to (0,\infty)$ such that the limit $\lim_{t\to 0+}p(t,x,x)\Psi (t)$ exists for all $x\in X$, we prove that for every potential $w:X\to \IR$ one has 
$$
\lim_{t \to 0+} \Psi (t)\mathrm{tr}\big(\mathrm{e}^{ -t   H_p   + w}\big)=  \int  \mathrm{e}^{-w(x) }\lim_{t \to 0+}p(t,x,x) \Psi (t) \Id\mu(x)<\infty
$$
for the Schrödinger type operator $H_p + w$, provided $w$ satisfies very mild conditions at $\infty$, that are essentially only made to guarantee that the sum of quadratic forms $ H_p   + w/t$ is self-adjoint and bounded from below for small $t$, and to guarantee that 
$$
 \int  \mathrm{e}^{-w(x) }\lim_{t\to 0+}p(t,x,x) \Psi (t) \Id\mu(x)<\infty.
$$
The proof is probabilistic and relies on a principle of not feeling the boundary for $p(t,x,x)$. In particular, this result implies a new semiclassical limit result for partition functions valid on arbitrary connected geodesically complete Riemannian manifolds, and one also recovers a previously established semiclassical limit result for possibly locally infinite connected weighted graphs. 

\end{abstract}

%\begin{center}
%\today
%\end{center}
%
%\wspace{2mm} 

\section{Introduction}

To motivate all following results, let us take a look at the case when $(X,g)$ is either the Euclidean $\IR^m$ or a closed Riemannian $m$-manifold, noting that for the sake of exposition we will not care about self-adjointness issues and other technicalitities for the moment. Given a function $f:\IR\to\IR$, classical results by Helffer/Robert \cite{helfer} (see also \cite{baer2} for closed $X$'s) show that for a sufficiently bounded smooth potential $w:X\to \IR$, one has
\begin{align}\label{bma}
\lim_{\hbar \to 0+} \f{ Z_{\mathrm{QM}}(g;w;f;\hbar)}{ Z_{\mathrm{cl}}(g;w;f;\hbar)} =1,
\end{align}
where
\begin{align*}
Z_{\mathrm{QM}}(g;w;f;\hbar):=\mathrm{tr}\big(f\big( \hbar^2  \Delta_g+w\big)\big)
\end{align*}
denotes the \lq\lq{}quantum partition function\rq\rq{} and  
$$
 Z_{\mathrm{cl}}(g;w;f;\hbar):=\f{1}{(2\pi\hbar)^m} \int_{T^*X} f\big(|p|^2_{g^*} +w(q)\big)  \Id p\wedge \Id q,
$$
denotes its classical analogue. Note that the integral in $Z_{\mathrm{cl}}(g;w;f;\hbar)$ is a globally well-defined integral of a differential form in $\Omega^{2m}(T^*X)$, as $\Id p\wedge \Id q$ is (up to a constant) just the coordinate representation of the $m$-th power of the symplectic form on $T^* X$.\\
The proofs of these results use an asymptotic expansion of $\hbar^{m}Z_{\mathrm{QM}}(g;w;f;\hbar)$ as $\hbar\to 0+$, and which naturally require strong global growth conditions on the potential $w$, in addition to smoothness. As their formulation is built upon cotangent bundles and forms, semiclassical limit results such as (\ref{bma}) for general $f$\rq{}s clearly cannot be formulated in other settings than manifold-like spaces. \\
 On the other hand, the most important choice of $f$ for quantum physics\footnote{Although other choices of $f$ are of course needed in quantum physics: For example, it has been pointed to the author by B. Helffer that the choice $f(\lambda):=S\log(1+\mathrm{e}^{-S/\lambda})$, $S>0$, is used in solid state physics (cf. \cite{helfer2} for a mathematical description of the De Haas - van Halfen effect).} and geometry is $f(\lambda):=\mathrm{e}^{-\lambda}$, where then the numerator in (\ref{bma}) becomes the usual quantum partition function from statistical physics and the denominator its classical analogue. Now (\ref{bma}) reflects the fact every quantum data having a classical analogue should by approximated by the latter as $\hbar\to 0+$. Moreover, in the exponential case the momentum integration (that is, the $p$-integration) in $Z_{\mathrm{cl}}(g;w;f;\hbar)$ factors and becomes a Gaussian integral, leading straightforwardly to the equivalence of (\ref{bma}) to  
\begin{align}\label{bma2}
\lim_{t \to 0+} (2\pi t)^{m/2} \mathrm{tr}\big(\mathrm{e}^{- t( \Delta_g+w/t)}\big)= \int_{X} \mathrm{e}^{-w} \Id\mu_g,
\end{align}
with $\Id\mu_g$ the Riemannian volume measure. This simple observation is used as a tool in a proof given in \cite{simon}, but it is actually much more than just a tool: the above reformulation of (\ref{bma}) does not refer to manifolds at all, as lang as one considers $\Delta_g$ abstractly as the generator of a strongly continuous self-adjoint contraction semigroup on an $L^2$-space. Looking for an abstract formulation of (\ref{bma2}), it clearly  remains to clarify the meaning of the regularizing factor $(2\pi t)^{m/2}$. Of course it is tempting to believe that this factor stems from the on-diagonal singularity of the Euclidean heat kernel. Another aspect of results like (\ref{bma2}) is completely hidden in the Riemannian setting, namely the choice of measure on the RHS of (\ref{bma2}). While in the Riemannian case it is the measure of the underlying Hilbert space, something else happens on weighted graphs: It has been shown recently in \cite{G45}, that for the canonically given self-adjoint Laplacian $\Delta_{b,\mu}$ in $\ell^2(X,\mu)$ on a possibly locally infinite connected weighted graph $(X,b,\mu)$ (cf. Example \ref{graph}) one has 
\begin{align}\label{graf} 
\lim_{t \to 0+}  \mathrm{tr}\big(\mathrm{e}^{-t  \Delta_{b,\mu}+w/t }\big)=  \sum_{x\in X}  \mathrm{e}^{-w(x)}.
\end{align}
What is seemingly different now from the Riemannian case (\ref{bma2}) is that on the right-hand side the integration is not with respect to the Hilbert space measure $\mu$. On the other hand, as for the heat kernel on graphs one has
$$
\lim_{t\to 0+}p_{b,\mu}(t,x,x)=1/\mu(x),
$$
one can rewrite the exponential integral according to
$$
 \sum_{x\in X}  \mathrm{e}^{-w(x)}=\sum_{x\in X}  \mathrm{e}^{-w(x)}\lim_{t\to 0+}p_{b,\mu}(t,x,x)\mu(x).
$$
This lead us to the following observations: Firstly, an abstract formulation of (\ref{bma2}) on a locally compact space $X$ with a well-behaved Radon measure $\mu$ should be built from an abstract $\mu$-heat kernel $p(t,x,y)$ (cf. Section \ref{Main}), replacing the Laplace operator with the self-adjoint generator $H_p$ of the semigroup induced by $p(t,x,y)$. Secondly, taking once more the Riemannian case (\ref{bma2}) into account, where one has 
$$
\lim_{t\to 0+}p_{g}(t,x,x)(2\pi t)^{m/2}=1,
$$
the analysis should be ultimately built assuming the existence of a function  $\Psi : (0,1)\to (0,\infty)$ such that the limit $
\lim_{t\to 0+}p(t,x,x)\Psi (t)$ exists for all $x\in X$. By what we have observed so far, the natural formulation of (\ref{bma2}) now becomes
\begin{align}\label{apodaa}
\lim_{t \to 0+} \Psi (t)\mathrm{tr}\big(\mathrm{e}^{ -t (   H_p   + w/t)}\big)=  \int  \mathrm{e}^{-w(x) }\lim_{t \to 0+}p(t,x,x) \Psi (t) \Id\mu(x). 
\end{align}
Without losing any important example, we will assume that $H_p$ stems from a regular Dirichlet form, and that $p(t,x,y)$ satisfies a probabilistic principle of not feeling the boundary (cf. the new Definition \ref{dpoa}).\vspace{1mm}

\emph{Our main result, Theorem \ref{main}, states that in this situation the formula (\ref{apodaa}) holds for every possibly unbounded continuous potential $w:X\to \IR$, under very mild global assumptions on $w$ that are only made to make all involved quantities well-defined at all.} \vspace{1mm}

As we have indicated above, Theorem \ref{main} is completely new even from a conceptually point of view. We believe it is a remarkable fact that it is possible to formulate such a result in a possibly nonlocal Dirichlet space setting, allowing to treat Riemannian manifolds and weighted graphs simultaniously. As a corollary to this result, we obtain a new semiclassical limit result for arbitrary complete connnected Riemannian manifolds (Corollary \ref{pmyy}), which is seemingly the first result of this type for noncompact manifolds other than $\IR^m$. In addition, Corollary \ref{pmyy} seems to be even new in the Euclidean $\IR^m$, in the sense that the conditions on the potential are weaker than in all previously known results. It is important to note that Corollary \ref{pmyy} does not impose any curvature conditions on the geometry. On the other hand this result can be significantly refined in case the Ricci curvature is bounded from below, as is shown in Corollary \ref{Ricci}, using the volume doubling machinery. Finally, we prove that Theorem \ref{main} allows to recover the formula (\ref{graf}) for graphs (Corollary \ref{graphh}).\\
We close this section with some remarks about the proof of Theorem \ref{main}: In \cite{baer2,helfer} the authors develop an asymptotic expansion of the underlying trace in order to prove the semiclassical limit (see also \cite{combes} for a different analytic approach). Our proof, on the other hand, is probabilistic and in the spirit of \cite{simon}, which treats the Euclidean Laplace operator. In fact, the machinery developed in this paper entails that it is actually sufficient to know the leading order term of the usual $t\to 0+$ asymptotic expansion of the \emph{unperturbed} heat kernel diagonal $\mathrm{e}^{-tH_p}(x,x)$, in the sense that then one can use path integrals as a perturbative tool to reduce everything to the latter expansion. One moral of this story is that, if one is only interested in its semiclassical limit for the trace, asymptotic expansions for the trace of the perturbed operator are an overkill. In addition, as indicated above, these expansions naturally require many strong local and global assumptions on the underlying data.\vspace{1mm}

We point out that the proof of Theorem \ref{main} becomes significantly more complicated than the one from \cite{simon}. The main reason for this is that the on-diagonal values of the Euclidean heat kernel $\mathrm{e}^{-t H_g}(x,x)=(2\pi t)^{m/2}$ do not depend on $x$ at all, and in addition one has many translation-invariance arguments that are based on Wiener measures. To compensate this lack of symmetries (which of course aren\rq{}t even present on arbitrary Riemannian manifold) in the general case, we were lead to the new Definition \ref{def1}, and we found that a principle of not feeling the boundary plays an essential role in the context of semiclassical limits of abstract Schrödinger semigroups. This seems to be a new observation. Another difficulty of our generalized setting is that the Hunt process associated with the regular Dirichlet form that stems from $p(t,x,y)$ has in general c\`{a}dl\`{a}g (and possibly explosive) paths, which makes some arguments more subtle. This is of course the price for working with possibly nonlocal Dirichlet forms.\vspace{3mm}

{\bf Acknowledgements:} I would like to thank Bernard Helffer for many very helpful remarks on the literature. In addition, I would like to thank Klaus Mohnke, Ralf Rueckriemen and Eren Ucar for helpful discussions. This research has been financially supported by the project SFB 647: Raum - Zeit - Materie of the German research foundation (DFG).

\section{Main results}\label{Main}

In the sequel, all function spaces are understood to be complex-valued. Let $X$ be a seperable metrizable locally compact space and let $\mu$ be a Radon measure on $X$ with a full support. We fix once for all a strictly positive pointwise consistent $\mu$-heat kernel
$$
p:(0,\infty)\times X\times X \longrightarrow (0,\infty),\quad (t,x,y)\longmapsto p(t,x,y).
$$ 
By definition, $p$ is a jointly Borel measurable function such that for all $t>0$, $x,y\in X$ one has
\begin{align}\label{A1}
&\>\> p(t+s,x,y)=\int  p(t,x,z)p(s,z,y) \Id\mu ( z),\\\label{A2}
&\>\> p(t,x,y)=p(t,y,x)  ,\\\label{A3}
&\>\>\int  p(t,x,z)\Id\mu( z)\leq 1,
\end{align}
such that with
$$
P_tf(x):=\int p(t,x,y) f(y)\Id\mu(z)
$$
it holds that
$$
\int\left|P_t f-f\right|^2\Id\mu\to 0\>\text{ as $t\to 0+$, for all $f\in\IL^2(X,\mu)$.}
$$
It follows that $(P_t)_{t>0}$ is a strongly continuous self-adjoint contraction semigroup in $\IL^2(X,\mu)$, so let $H_p\geq 0$ denote its self-adjoint generator, and $Q_p\geq 0$ the densely defined symmetric sesquilinear form associated with $H_p$.

\begin{Definition}\label{dpm} $p$ is called \emph{regular}, if $C_{\c }(X)\cap\dom(Q_p)$ is dense in $C_{\c}(X)$ with respect to the uniform norm, and dense in $\dom(Q_p)$ with respect to the norm $Q_p(f,f)+\int |f|^2\Id \mu$.
\end{Definition}

If $p$ is regular, then $Q_p$ becomes a regular Dirichlet form \cite{fuku} in $\IL^2(X,\mu)$, a fact that should justify this notion.\vspace{2mm}

Given $t>0$ let $\Omega(X,t)$ denote the space of c\`{a}dl\`{a}g paths $\gamma: [0,t]\to X$. For every $0\leq s\leq t$ let $\IFF_s(X,t)$ denote the sigma-algebra on $\Omega(X,t)$ which is generated by the coordinate map on $\Omega(X,t)$ including the time $s$, and let the sub-$\sigma$-algebra $\IFF^-_s(X,t)\subset \IFF_s(X,t)$ be defined by
$$
\IFF^-_s(X,t):=\text{smallest-$\sigma$-algebra on $\Omega(X,t)$ which contains $\bigcup_{0\leq u<s} \IFF_u(X,t)$}  .
$$
So $\IFF^-_s(X,t)$, $0\leq s\leq t$, is the canonically given left-continuous filtration of $\IFF_t(X,t)$. If $p$ is regular, then by a classical result of Fukushima for every regular Dirichlet form there exists a possibly explosive Hunt process \cite{fuku} associated with its semigroup (which is $(P_t)_{t>0}$ in our case). This immediately gives the first part of the following result, where $\IP^{x_0}_t$ is just the restriction of the law of this process with initial point $x_0$ to paths that do not explode until $t$:

\begin{Propandef}\label{wiener} Assume that $p$ is regular.\\
\emph{a)} For every $x_0\in X$, \emph{the Wiener measure $\IP^{x_0}_t$ from $x_0$ with terminal time $t$ w.r.t. $p$} is defined to be the unique sub-probability measure on $ \IFF_t(X,t) $ which satisfies
\begin{align}\label{pamf}
&\IP^{x_0}_t\{\gamma:\gamma(t_1)\in A_1,\dots,\gamma(t_n)\in A_n\} \\\nn
&=\int_{A_1}\cdots\int_{A_n} p(\delta_0 ,x_0,x_1) \cdots p(\delta_{n-1} ,x_{n-1},x_n) \Id \mu(x_1)\cdots \Id\mu(x_n)
\end{align}
for all $n\in\IN_{\geq 1}$, all partitions $0=t_0<t_1< \dots<t_{n-1}<t_n=t$ and all Borel sets $A_1,\dots,A_n\subset X$, where $\delta_j:=t_{j}-t_{j-1}$.\\
\emph{b)} For every $x_0,y_0\in X$, the \emph{pinned Wiener measure $\IP^{x_0,y_0}_t$ from $x_0$ to $y_0$ with terminal time $t$ w.r.t. $p$} is defined to be the unique probability measure on $ \IFF^-_t(X,t) $ which satisfies
\begin{align}\label{aposs}
 \IP^{x_0,y_0}_t(A)=\f{1}{p(t,x_0,y_0)}\int_A p(t-s,\gamma(s),y_0)\Id \IP^{x_0}_t(\gamma)
\end{align}
for all $0\leq s<t$ and all $A\in \IFF_s(X,t)$.
\end{Propandef}

Note that $\IP^{x_0}_t$ is in general only a sub-probability measure, as we have made no conservativeness assumption. The second part of Proposition \ref{wiener} follows from the results of \cite{palm}. Note the subtlety that the pinned Wiener measure is only defined $\IFF_t^-(X,t)$, which is in general not equal to $\IFF_t(X,t)$, as the paths are not left-continuous.

% entail that for every $x_0,y_0\in X$, $t>0$ there exists a unique probability measure $\IP^{x_0,y_0}_t$ on $ \IFF^-_t(X,t)$ which satisfies
%\begin{align*}
%&\IP^{x_0,y_0}_t \{\gamma :\gamma(t_1)\in A_1,\dots,\gamma(t_n)\in A_n\} \\
%&=\f{1}{p(t,x,y)}\int_{A_1}\cdots\int_{A_n} p(t_1-t_0,x_0,x_1) \cdots p(t-t_{n} ,x_{n},y) \Id \mu(x_1)\cdots \Id\mu(x_n)
%\end{align*}
%for all $n\in\IN_{\geq 1}$, all partitions $0=t_0<t_1< \dots<t_n<t$ and all Borel sets $A_1,\dots,A_n\subset X$. The measure $\IP^{x_0,y_0}_t $ is called the \emph{pinned Wiener measure corresponding to $p$ with starting point $x_0$, terminal time $t$ and terminal point $y_0$.}

%\begin{Propandef} It follows from the results of \cite{palm} that for every $t>0$, $x_0,y_0\in X$ there exists a unique probability measure on $\IFF(X,t;t)$ which satisfies
%\begin{align}\label{apoqqq}
 %\IP^{x_0,y_0,t}_H(A)=\f{1}{\mathrm{e}^{-t H}(x_0,y_0)}\int_A \mathrm{e}^{-(t-s) H}( \gamma(t),y_0)\Id \IP^{x_0}_t(\gamma)\>\text{ for all $0\leq s<t$, and all $A\in \IFF(X,t;s)$.}
%\end{align}
%The measure $\IP^{x_0,y_0,t}_H$ is called the \emph{pinned Wiener measure corresponding to $H$, with starting point $x_0$, terminal time $t$ and terminal point $y_0$.}
%\end{Propandef}

Now we can give the following definition:

\begin{Definition}\label{dpoa} If $p$ is regular, then it is said to satisfy \emph{the principle of not feeling the boundary}, if for all compact subsets $K\subset X$ having a nonempty interior $\mathring{K}$ and all $x\in \mathring{K}$, the pinned Wiener measures satisfy
$$
\lim_{t\to 0+}\mathbb{P}^{x,x}_t    \{\gamma: \ \gamma(s)\in K\text{ for all $s\in [0,t)$}\}=1.
$$
\end{Definition}

Note that in the situation of Definition \ref{dpoa} one indeed has
$$
 \{\gamma: \ \gamma(s)\in K\text{ for all $s\in [0,t)$}\}\in \IFF^-_t(X,t),
$$
as can be seen from
$$
\Omega(X,t)\setminus \{\gamma: \ \gamma(s)\in K\text{ for all $s\in [0,t)$}\}= \bigcup_{s\in [0,t)\cap \mathbb{Q}}\{\gamma: \gamma(s)\in X\setminus K\}.
$$ 
It is shown below that the principle of not feeling the boundary holds on arbitrary complete Riemannian manifolds and weighted graphs. We believe that an abstract investigation of this property should be of an independent interest. \vspace{2mm}

Let us now take a look at the perturbations of $H_p$ by potentials $w:X\to\IR$ that we will consider in the sequel. The next (well-known) definition is motivated by the fact that we want to investigate the behaviour of operators of the form $\exp(-( t H_p+w))$, where $t >0$ is small. In order to use the Feynman-Kac formula, which is valid for expressions of the form $\exp(-t ( H_p+w\rq{}))$, we have to factor 
$$
\exp(-( t H_p+w))=\exp(-t ( H_p+w/t ))
$$
and this shows that we have to guarantee that $H_p+w/t $ is semibounded from below for all small $t\rq{}s$. This clearly requires some control on the negative part of $w$, which is the simple idea behind the following definition, that will implement the above in a general quadratic form sense:

\begin{Definition} We say that a Borel function $w:X\to \IR$ is in the \emph{infinitesimal $Q_p$-class} $\mal_p(X,\mu)$, if for all $\epsilon>0$ there exists an $C_{\epsilon}\geq 0$ such that for all $f\in\dom(Q_p)$ one has 
$$
\int |w| |f|^2 \Id\mu \leq \epsilon Q_p(f,f)+ C_{\epsilon}\int |f|^2 \mu<\infty. 
$$
%and that $w$ is in the \emph{local Kato class} $\mal_{\loc}(X,p)$, if $1_Kw\in \mal(X,p)$ for all compact $K\subset X$. \\
%2. We say that $w$ is in the class of \emph{strictly Kato decomposable potentials} $\mal_{\pm}(X,p)$, if there exist potentials $w_{\pm}:X\to [0,\infty)$ such that 
%\begin{itemize} \item 
%$w=w_+-w_-$
%\item $w_+\in \mal_{\loc}(X,p)$
%\item $w_-$ is in the Kato class $\mal(X,p)$.
%\end{itemize}
\end{Definition}

Clearly, this property depends only on the $\mu$-equivalence class of $w$. It is also important to note that bounded functions are in $\mal_p(X,\mu)$. In case $p$ is regular, the results from \cite{peter} also show that $\mathcal{K}_{p} (X,\mu)\subset \mal_p(X,\mu)$, where $\mathcal{K}_p(X,\mu)$ stands for the Kato class with respect to $p$ and $\mu$. Recall here that a Borel function $w:X\to\IR$ is by definition in $\mathcal{K}_{p} (X,\mu)$, if and only if
$$
\lim_{t\to 0 +}\int_{  X}\int^t_0p(s,x,y) |w(y)|\Id\mu(y)\Id s=0.
$$
Finally, (possibly weighted) $L^q$-conditions on $w$ that imply $w\in \mathcal{K}_{p} (X,\mu) $ can be found in \cite{kt,gun}.\vspace{1mm}

Given a function $w:X\to\IR$ let $w^{\pm}=\max(\pm w,0)$ denote its positive and negative part respectively, so that $w=w^+-w^-$.

\begin{Propandef} Let $p$ be regular and assume we are given a potential $w:X\to \IR$ with $1_Kw\in L^1(X,\mu)$ for all compact $K\subset X$ and $w^-\in \mal_p(X,\mu)$. Then the symmetric sesqui-linear form
\begin{align}\label{apa}
 Q_{p}(w)(f_1,f_2):=Q_p( f_1, f_2)+\int  \overline{w}f_1 f_2 \Id\mu 
\end{align}
with domain of definition given by all $f\in\dom(Q_p)$ such that $\int  |w||f|^2\Id\mu<\infty$, is densely defined, semibounded (from below) and closed in $\IL^2(X,\mu)$. The semibounded self-adjoint operator in $\IL^2(X,\mu)$ associated with the form $Q_{p}(w)$ will be denoted with $H_p(w)$. 
\end{Propandef}

\begin{proof} The claim for the case $w_+=0$ follows immediately from the KLMN theorem. For the general case, one just has to note that the maximally defined form given by $w_+$ is densely defined and closed (using Fatou\rq{}s lemma). Thus (\ref{apa}) is just the sum of closed symmetric forms and thus has these properties, too. In addition, the form is densely defined, as it clearly contains $\dom(Q_p)\cap C_{\c}(X)$, and the latter set is dense in $\IL^2(X,\mu)$, by the regularity of $p$. 
\end{proof}

In order to formulate our main result, we add:

\begin{Definition}\label{def1}  A pair $(\varrho_1,\varrho_2)$ of Borel functions
$$
\varrho_1: (0,1)\longrightarrow (0,\infty),\quad \varrho_2:X\longrightarrow [0,\infty)
$$
is called an \emph{asymptotic control pair for $p$}, if the following assumptions are satisfied:
\begin{itemize}
%\item one has $\inf \varrho_1>0$
\item the limit $\lim_{t\to 0+}p(t,x,x)\varrho_1(t)$ exists for all $x\in X$.

\item there exists a Borel function $\phi:(0,1)\to (0,\infty)$ such that 
\begin{align*}
&p(t, x,x)\lesssim \varrho_2(x)\phi(t)\>\>\text{ for all $(t,x)\in (0,1)\times X$,}\quad \sup_{0<t<1}\phi(t)\varrho_1(t)<\infty.
\end{align*}
\end{itemize}
\end{Definition}

There is no reason to expect that every strictly positive pointwise consistent heat kernel admits an asymptotic control pair. However, in the setting of Riemannian manifolds and weighted infinite graphs one has several canonical choices, also without any further assumptions on the geometry:

\begin{Example}\label{dpoaaa} Assume that $(X,g)$ is a smooth geodesically complete connected Riemannian $m$-manifold, $\mu_g$ the associated Riemannian volume measure, and $-\Delta_g=\Id^g\Id$ the Laplace-Beltrami operator. Then for every fixed $y\in X$ the minimal nonnegative solution $(t,x)\mapsto p_g(t,x,y)$ of 
$$
\partial_tu(t,x)= \Delta_g u(t,x),\quad \lim_{t\to 0+} u(t,\bullet)=\delta_y\quad\text{ as distributions}
$$
in $(0,\infty)\times X$ induces a strictly positive pointwise consistent $\mu_g$-heat kernel $p_g$. Now by locality the Wiener measures are concentrated on continuous paths, and the principle of not feeling the boundary follows easily from Theorem 1.2 in \cite{Hsu2}. In this situation, the operator $H_g:=H_{p_g}$ is the unique self-adjoint extension of $-\Delta_g$ (this uniqueness requires geodesically completeness), and we set 
$$
\mal(X,g):=\mal_{p_g}(X,\mu_g),\quad  \mathcal{K}(X,g):=\mathcal{K}_{p_g}(X,\mu_g).
$$
In case $w: X\to \IR$ is a continuous potential  with $w^-\in \mathcal{K}(X,g)$ then the definition of $H_p(w)$ becomes very easy, as then $-\Delta+w$ is essentially self-adjoint \cite{post}. \\
In leading order, the asymptotic expansion of $p_g(t,x,x)$ as $t\to 0+$ (cf. \cite{Hsu2} for a detailed proof in the noncompact geodesically complete case) implies
$$
\lim_{t\to 0+}p_g(t,x,x)\varrho^{(m)}(t)=1,
$$
where $\varrho^{(m)}(t):=(2\pi t)^{m/2}$. Given $x\in X$, and $b>1$, let $r_{\mathrm{Eucl},g}(x,b)$ be the supremum of all $r>0$ such that the open ball $B_g(x,r)$ with respect to the geodesic distance admits a coordinate system 
$$
\vec{x}=(x^1,\dots,x^m):B_g(x,r)\longrightarrow U\subset \IR^m
$$
with $\vec{x}(x)=0$ and with respect to which one has one has the following inequality for all $y\in B_g(x,r)$, 
\begin{align}\label{equi}
&(1/b)(\delta_{ij})\leq (g_{ij}(y)):= \left(g(\partial_i,\partial_j)(y)\right)\leq b (\delta_{ij})\>\text{ as symmetric bilinear forms.}
\end{align}
Then $r_{\mathrm{Eucl},g}(x,b)$ is called \emph{the Euclidean radius of $X$ at $x$ with accuracy $b$}. It is shown in \cite{gun} that for all $b>1$ the function
\begin{align}\label{apn}
\varrho_g(x):= 1/\min (r_{\mathrm{Eucl},g}(x,2), 1)^{m} , \quad x\in X,
\end{align}
turns $(\varrho^{(m)},\varrho_g)$ into an asymptotic control function for $p_g$, without any curvature assumptions on $(X,g)$. This result relies on a parabolic $\IL^1$-mean value inequality for solutions of the heat equation.\\
In case the Ricci curvature $\mathrm{Ric}_g$ is bounded from below by a constant, then the function
$$
\varrho_g'(x):=  1/\mu_g(B_g(x,1)), \quad x\in X,
$$
turns $(\varrho^{(m)},\varrho_g')$ into an asymptotic control function for $p_g$ (cf. Example 2.7 in \cite{gun}). This follows from the Li-Yau heat kernel estimate and volume doubling.
\end{Example}
 
As for graphs:

\begin{Example}\label{graph} Recall that a weighted graph is a triple $(X,b,\mu)$, with $X$ is a countable set, $b$ a symmetric function
$$
b: X\times X\longrightarrow  [0,\infty) \text{ with $b(x,x)=0$, $\sum_{y\in X} b(x,y)<\infty$ for all $x\in X$,}
$$
and $\mu:X\to (0,\infty)$ is an arbitrary function. One equipps $X$ with its discrete topology (which is thus induced by the discrete metric) and writes $x\sm  y$ if $b(x,y)>0$, referring to such points as the edges of graph. In this sense, the points of X become the vertices and $\mu$ a vertex weight function. The function $\mu$ defines a Radon measure having a full support  in the obvious way:
$$
\mu(A)=\sum_{x\in A}\mu(x), \>\>A\subset X.
$$ 
In particular, the scalar product on $\IL^2(X,\mu)$ is given by $\sum_{x\in X}\overline{f_1(x)}f_2(x)\mu(x)$. \\
We assume also that $(X,b)$ is connected in the graph theoretic sense, meaning that for any $x,y\in X$ there is a finite sequence $x_1,\dots,x_n\in X$ such that $x_0=x$, $x_n=y$. \\
Denoting the space of complex-valued functions on $X$ with $C(X)$, where an index \lq{}$\mathrm{c}$\rq{} now simply means \lq{}finitely supported\rq{}, and a set of functions $F_{b}(X)$ given by all $\psi\in C(X)$ such that $\sum_{y\in X}b(x,y)|\psi(y)|<\infty$  for all $x\in X$, we define a Laplace type formal difference operator $\Delta_{b,\mu}$ by
\begin{align}\label{formal}
\Delta_{b,\mu}:F_{b}(X)\longrightarrow C(X),\>\Delta_{b,\mu}\psi(x)=-\f{1}{\mu(x)}\sum_{\{y:y\sm x\}}b(x,y)\big(\psi(x)-\psi(y)\big).
\end{align}

Then for all fixed $y\in X$, there exists \cite{kl} a pointwise minimal element $p_{b,\mu}(\cdot,\cdot,y)$ of the set given by all bounded functions 
$$
u:(0,\infty)\times X\to [0,\infty)
$$
that satisfy 
$$
\partial_t u(t,x)=\Delta_{b,\mu} u(t,x),\>\> \lim_{t\to 0+}u(t,x)=\delta_y(x)\quad\text{for all $(t,x)\in (0,\infty)\times X$.}
$$
The function $(t,x,y)\mapsto p_{b,\mu}(t,x,y)$ is strictly positive (this property requires the above notion of connectedness) and defines a strictly positive pointwise consistent and regular $\mu_{\mu}$-heat kernel \cite{kl}. In addition, one has the trivial estimates
$$
p_{b,\mu}(t,x,y)\leq 1/\mu(x)\quad\text{ for all $(t,x,y)\in (0,\infty)\times X\times X$},
$$
and the convergence
$$
\lim_{t\to 0+}p_{b,\mu}(t,x,x)\mu(x)=1 \quad\text{ for all $x\in X$.}
$$
The first property follows from $\sum_x p(t,x,y)\mu(y)\leq 1$, and the second one from the strong continuity of $P_t$ (noting by discreteness, $\IL^2$-convergence implies pointwise convergence). To see that $p_{b,\mu}$ satisfies the principle of not feeling the boundary, note that in this case $\mathbb{P}^{x,x}_t$ is concentrated on pure jump paths, and one has
\begin{align*}
&\mathbb{P}^{x,x}_t\{\gamma: \gamma(s)\in\{x,x_1,\dots, x_l\}\text{ for all $s\in [0,t)$}\}\geq \>\mathbb{P}^{x,x}_t\{\gamma: \text{$\gamma$ has not jumped before $t$}\}\\
&\geq  \exp\Big(-\f{t}{\mu(x)}\sum_y b(x,y) \Big) p(t,x,x)^{-1}\mu(x)^{-1},
\end{align*}
where the first estimate is trivial and the second one is well-known (cf. page 10 in \cite{G45} and the references therein). Finally, in this case the operator $H_{b,\mu}:= H_{p_{b,\mu}}$ is a restriction of $\Delta_{b,\mu}$ \cite{kl, GKS-13,G45}, and we set $\mal(X,b,\mu):=\mal_{p_{b,\mu}}(X,\mu)$.

%$\mathbb{P}^{x,y}_t$ now extends to $\mathbb{P}^{x_j,x_j}_t$
%
%
%just note that given a finite set $\{x_1,\dots, x_l\}\subset X$ and $j\in \{1,\dots,l\}$ one has 
%\begin{align*}
%&\mathbb{P}^{x_j,x_j}_t\{\gamma: \gamma(s)\in\{x_1,\dots, x_l\}\text{ for all $s\in [0,t)$}\}\\
%&\geq \mathbb{P}^{x_j,x_j}_t\{\gamma: \text{}\}
%\end{align*}
\end{Example}

Here is our main result:

\begin{Theorem}\label{main} Assume that there exists a metric on $X$ which induces the original topology in a way that for every $x\in X$ there is an $r>0$ such that the open metric ball $B(x,r )$ is relatively compact. Assume also that $p$ satisfies the principle of not feeling the boundary. Then for every asymptotic control pair $(\varrho_1,\varrho_2)$ for $p$, and every continuous potential $w:X\to \IR$ with $w^-\in \mal_p(X,\mu)$ and $\int  \mathrm{e}^{-w} \varrho_2\Id\mu<\infty$, one has
\begin{align}\label{aappp}
\lim_{t \to 0+} \varrho_1(t) \mathrm{tr}\big(\mathrm{e}^{-t  H_p(w/t)}\big)=  \int  \mathrm{e}^{-w(x)}\lim_{t \to 0+}p(t,x,x)\varrho_1(t)\Id\mu(x)<\infty.
\end{align}
\end{Theorem}

The proof of Theorem \ref{main} will be given in Section \ref{dsaar} below. A few remarks about the formulation of Theorem \ref{main} are in order:

\begin{Remark} 1. The integrability 
$$
\int  \mathrm{e}^{-w(x)}\lim_{t \to 0+}p(t,x,x)\varrho_1(t)\Id\mu(x)<\infty
$$
is not assumed, but is a consequence of the assumptions. Clearly, this implies that for all sufficiently small $t$ one has $\mathrm{tr}\big(\mathrm{e}^{-t  H(w/t )}\big)<\infty$, showing that $H_p(w/t )$ has a discrete spectrum for small $t$'s.\\
2. The continuity of $w$ is only used to prove the lower bound  
$$
\liminf_{t \to 0+} \varrho_1(t) \mathrm{tr}\big(\mathrm{e}^{-t  H_p(w/t)}\big)\geq   \int  \mathrm{e}^{-w(x)}\lim_{t \to 0+}p(t,x,x)\varrho_1(t)\Id\mu(x),
$$
and for this to hold an inspection of the proof shows that it is enough to assume that $w$ is continuous away from a closed set $N$ with $\mu(N)=0$.\\
The upper bound 
$$
\limsup_{t \to 0+} \varrho_1(t) \mathrm{tr}\big(\mathrm{e}^{-t  H_p(w/t )}\big)\leq  \int  \mathrm{e}^{-w(x)}\lim_{t \to 0+}p(t,x,x)\varrho_1(t)\Id\mu(x)
$$
remains true under much weaker local assumptions on $w$ than continuity. For example, an assumption of the form $1_Kw\in\IL^1(X,\mu)\cap \mathcal{K}_{p} (X,\mu)$ for all compact $K\subset X$ would do.
\end{Remark}

Let us specialize our main result to the Riemannian setting (cf. Example \ref{dpoaaa} for the corresponding notation). Recall that  $\varrho^{(m)}(t):=(2\pi t)^{m/2}$.  

\begin{Corollary}\label{pmyy} Assume that $(X,g)$ is a smooth geodesically complete connected Riemannian $m$-manifold. Then for every Borel function $\varrho :X\to [0,\infty)$ which makes $  (\varrho^{(m)},\varrho )$ an asymptotic control pair\footnote{For example, one can take $\varrho:=\varrho_g$ with $\varrho_g$ as in (\ref{apn}).} for $p$, and for every continuous potential $w:X\to \IR$ with $w^-\in \mal(X,g)$ and $\int \mathrm{e}^{-w} \varrho \Id\mu_g <\infty$, one has
\begin{align*}
\lim_{t \to 0+} (2\pi t)^{m/2} \mathrm{tr}\big(\mathrm{e}^{-t  H_g(w/t)}\big)=  \int  \mathrm{e}^{-w(x) }\Id\mu_g(x)<\infty.
\end{align*}
\end{Corollary}

In case $\mathrm{Ric}_g$ is bounded from below by a constant, one can turn the integrability assumption from Corollary \ref{pmyy} on $w$ into a natural pointwise one:

\begin{Corollary}\label{Ricci} Let $(X,g)$ be a smooth geodesically complete connected Riemannian $m$-manifold with $\mathrm{Ric}_g\geq -A $ for some constant $A \geq 0$, and let $w:X\to \IR$ be a continuous potential with $\inf_X w>-\infty$ and\footnote{Every continuous $w:X\to \IR$ with $w(x)\geq A\rq{}d_g(x,x_0)^2$ for all $x\in X$ and some $A\rq{}>0$ satisfies (\ref{summ}). A similar \emph{linear} growth condition from below suffices, too, as long as $A\rq{}>A$.}
\begin{align}\label{summ}
\sum_{k=2}^{\infty}\exp\Big(-\inf_{x\in X, k-1<d_g(x,x_0)<k}w(x)\Big) k^m\mathrm{e}^{2k\sqrt{(m-1)A}} <\infty\quad\text{ for some $x_0\in X$.}
\end{align}
Then one has
$$
\lim_{t \to 0+} (2\pi t)^{m/2}\mathrm{tr}\big(\mathrm{e}^{-t  H_g(w/t)}\big)=  \int \mathrm{e}^{-w(x)}\Id\mu_g(x)<\infty.
$$ 
\end{Corollary}

The proof of Corollary \ref{Ricci} will be given in Section \ref{Rio}.\vspace{2mm}

Concerning graphs, our main result allows to recover the following result from \cite{G45}, which follows immediately from the considerations of Example \ref{graph}, where remarkably to \lq\lq{}local finiteness\rq\rq{} assumptions on the graph have to be imposed:

\begin{Corollary}\label{graphh} Let $(X,b,\mu)$ be a weighted graph which is connected in the graph theoretic sense. Then for every potential $w:X\to \IR $ with $w^-\in\mal (X,b,\mu)$ and $\sum_{x\in X}\mathrm{e}^{-w(x)}<\infty$, one has
\begin{align*} 
\lim_{t \to 0+}  \mathrm{tr}\big(\mathrm{e}^{-t  H_{b,\mu}(w/t )}\big)=  \sum_{x\in X}  \mathrm{e}^{-w(x)}.
\end{align*}
\end{Corollary}

\section{Proof of Theorem \ref{main} }\label{dsaar}
  
%Let us remark that for every con$w\in \mal_{\loc}(X)$ one has \cite{G2}
%
%\begin{align}\label{brw}
%\mathbb{P}^x_t\Big\{\gamma: \int^t_0|w(\gamma(s))|\Id s =\infty\Big\}=0\quad\text{ for every $x\in X$, $t>0$.}
%\end{align}
%Upon writing $\int^t_0|w(\gamma(s))|\Id s$ as $\int^{t/2}_0|w(\gamma(s))|\Id s+\int_{t/2}^t |w(\gamma(s))|\Id s$ and using the time reversal property\footnote{that is, the pushforward of $\mathbb{P}^{x,y}_t$ with respect to to the map $\Omega^t_X\to \Omega^t_X$, $\gamma\mapsto \gamma(t-\bullet)$ is $\mathbb{P}^{y,x}_t$ } of the pinned Wiener measure \cite{semi}, this also implies (\ref{brw}) with $\mathbb{P}^x_t$ replaced by $\mathbb{P}^{x,y}_t$, for every $x,y\in X$, $t>0$. In addition, for every $w\in \mal(X)$ one has
%\begin{align}\label{brw2}
%\int   \mathrm{e}^{-\int^{t}_0w(\gamma(s))\Id s}  \Id\mathbb{P}^{x}_{ t }(\gamma)<\infty\quad\text{ for every $x\in X$, $t>0$},
%\end{align}
%which with the same time reversal argument also implies (\ref{brw2}) with $\mathbb{P}^x_t$ replaced by $\mathbb{P}^{x,y}_t$, for every $x,y\in X$, $t>0$. \\

For the sequel, we record that the Riemann integrals $\int^{t}_0w(\gamma(s))\Id s$ for $w\in C(X)$, $\gamma\in \Omega(X,t)$, are well-defined, equal to their Lebesgue counterparts, and $\gamma\mapsto \int^{t}_0w(\gamma(s))\Id s$ is $\IFF^-_t(X,t)$-measurable. The following Feynman-Kac formula for the trace of a Schrödinger type semigroup will be the main tool for the proof of Theorem \ref{main}:

\begin{Theorem}\label{awpo} Let $p$ be a regular $\mu$-heat kernel. Then for every continuous potential $w:X\to \IR$ with $w^-\in \mal_p(X,\mu)$ and every $t>0$ one has
\begin{align}\label{pqn}
 \mathrm{tr}\left(\mathrm{e}^{-t H_p(w)}\right)= \int  p(t,x,x) \int   \mathrm{e}^{-\int^{t}_0w(\gamma(s))\Id s}  \Id\mathbb{P}^{x,x}_{ t }(\gamma) \ \Id\mu(x)\in [0,\infty]. 
\end{align}
\end{Theorem}

\begin{proof} Note first that by a monotone class argument, the function
$$
X\times X \ni (x,y) \longmapsto \mathbb{P}^{x,y}_{ t }(A)\in \IR
$$
is jointly Borel measurable for every $A\in \IFF^-_t(X,t)$, so that the integral on the RHS of (\ref{pqn}) is well-defined. Let us start with the Feynman-Kac formula
\begin{align}\label{sta0}
\mathrm{e}^{- tH_p(w)}f(x)=\int   \mathrm{e}^{- \int^{t}_0w(\gamma(s))\Id s}  f(\gamma(t))\Id\mathbb{P}^{x}_{t }(\gamma),
\end{align}
valid for all $f\in \IL^2(X,\mu)$, and $\mu$-a.e. $x\in X$. In view of the remarks prior to Theorem \ref{awpo}, formula (\ref{sta0}) can be proved precisely as in the case of Riemannian manifolds \cite{simon,habil,G2}. We sketch a proof: First one assumes that $w$ is bounded. Then Trotter\rq{}s formula and (\ref{pamf}) easily imply (\ref{sta0}). Then one assumes $\inf w>-\infty$, and applies the previous result to $w_n:=\min(w,n)$, $n\in \IN$, using monotone convergence of quadratic forms to control the LHS (\ref{sta0}) and convergence theorems for integrals to control the RHS. Finally, for general $w$\rq{}s, one considers $w_n:= \max(-n,w)$ with an analogous reasoning.\\
The reader should note that under the given assumption on $w^-$ one only has
$$
\int   \mathrm{e}^{- \int^{t}_0w(\gamma(s))\Id s} |f(\gamma(t))|\Id\mathbb{P}^{x}_{t }(\gamma)<\infty\quad\text{for $\mu$-a.e. $x$},
$$
and \emph{not for all $x$}, while a Kato assumption on $w^-$ would guarantee this integrability for all $x$.\\
 An important observation now is that Hunt processes are quasi-left continuous, and thus almost surely cannot jump at a fixed time $t>0$, in the sense that
$$
\mathbb{P}^{x}_{t }\{\gamma: \gamma(s)\ne \gamma(s-)\}=0\quad\text{ for all $x\in X$, $0<s\leq t$}.
$$
Thus one can rewrite (\ref{sta0}) according to
\begin{align}\label{sta}
\mathrm{e}^{- tH_p(w)}f(x)=\int   \mathrm{e}^{- \int^{t}_0w(\gamma(s))\Id s}  f(\gamma(t^-))\Id\mathbb{P}^{x}_{t }(\gamma),
\end{align}
the integrand now being $\IFF^-_t(X,t)$ measurable. Using the disintegration formula \cite{palm}
\begin{align*}
\int F(\gamma) \Id\mathbb{P}^{x}_t(\gamma)=\int  p(t,x,y) \int F(\gamma) \Id\mathbb{P}^{x,y}_t(\gamma) \Id\mu(y),
\end{align*}
valid for all $t>0$, $x,y\in X$ and all $\IFF^-_t(X,t)$-measurable functions $F:\Omega\to [0,\infty]$, and
Formula (\ref{sta}) in combination with the normalization property \cite{palm}
$$
\mathbb{P}^{x,y}_t  \{\gamma: \gamma(t^-)=y\}=1\quad\text{for all $t>0$, $x,y\in X$,}
$$
implies that for every $t>0$ the operator $\mathrm{e}^{-t H_p(w)}$ is an integral operator with an integral kernel which for all $(x,y)$ is well-defined by
$$
\mathrm{e}^{-t H_p(w)}(x,y):=p(t,x,y)\int   \mathrm{e}^{- \int^{t}_0w(\gamma(s))\Id s}  \Id\mathbb{P}^{x,y}_{t }(\gamma)\in [0,\infty].
$$
This quantity is finite for $\mu$-a.e. $(x,y)$. The claimed formula now follows from the standard fact
\begin{align*}
\mathrm{tr}\left(\mathrm{e}^{-t H_p(w)}\right)=\int \int  \mathrm{e}^{-(t/2) H_p}(x,y)\mathrm{e}^{-(t/2) H_p(w)}(y,x) \Id\mu(y)\Id\mu(x)
\end{align*}
for semigroups given by integral kernels, in combination with 
\begin{align}\label{gott}
\int  \mathrm{e}^{-(t/2) H_p(w)}(t/2,x,y)\mathrm{e}^{-(t/2) H_p(w)}(t/2,y,x) \Id\mu(y)=\mathrm{e}^{-  H_p(w)}(x,x)
\end{align}
for all $x,y\in X$, $t>0$. The latter pointwise semigroup property follows in case $\inf w>-\infty$ precisely as in the Euclidean situation (cf. page 15 in \cite{sz}) from the Markoff property \cite{palm} of the Wiener measures. In the general case, the semigroup property thus holds for $w_n:=\max(-n,w)$. In order to take $n\to \infty$, let us record the following generalized convergence theorem for integrals: if $\phi_n\geq \phi$, $\phi_n\leq h$, $\int |h|<\infty$, $\phi_n\to \phi$, then $\int \phi_n\to \int \phi\in [0,\infty]$. This convergence theorem easily implies 
$$
\lim_{n\to\infty}\int^r_0 w_n(\gamma(s))\Id s\to \int^r_0 w_n(\gamma(s))\Id s \quad\text{ for all $r>0$, $\gamma \in \Omega(X,r)$, }
$$
and finally the $\lim_{n\to\infty}$ can be interchanged with the various Wiener integrals and the $\mu$-integration appearing in (\ref{gott}) using monotone convergence. 
\end{proof}

%The point of the latter result is that it does not need the joint continuity of
%$$
%X\times X\ni (x,y)\longmapsto \mathrm{e}^{-t H}(x,y) \int   \mathrm{e}^{-\int^{t}_0w(\gamma(s))\Id s}  \Id\mathbb{P}^{x,y}_{ t }(\gamma)\in (0,\infty).
%$$
%In fact, it is not clear to the author whether on a general Riemannian manifold the assumptions of Theorem \ref{awpo} imply the joint continuity of this map. The continuity in question is, however, is well-known to hold under some mild control on the geometry of $X$. To see that, one can mimick the Euclidean proof from \cite{sz}, which even shows joint continuity in $(t,x,y)$.\vspace{1mm}

Now we can give the proof of Theorem \ref{main}:

\begin{proof}[Proof of Theorem \ref{main}] We start with the\vspace{2mm}

\emph{proof of the upper bound }
\begin{align}\label{h1}
\limsup_{t \to 0+}\varrho_1(t) \mathrm{tr}\big(\mathrm{e}^{-t  H_p(w/t )}\big)\leq \int  \mathrm{e}^{-w(x)}\lim_{t\to 0}p(t,x,x)\varrho_1(t)  \Id\mu(x)<\infty.
\end{align}
 The abstract Golden-Thompson inequality (Corollary 3.9 in \cite{nu}) states that given two semibounded (from below) self-adjoint operators $H_1$, $H_2$ on a Hilbert space one has
$$
\mathrm{tr}\left(\mathrm{e}^{-t (H_1\dotplus H_2)}\right)\leq \mathrm{tr}\left(\mathrm{e}^{-(t/2) H_1}\mathrm{e}^{-(t/2) H_2}\mathrm{e}^{-(t/2) H_1}\right)\quad\text{ for all $t\geq 0$,}
$$
where $H_1\dotplus H_2$ denotes the form sum. Using the Chapman-Kolmogoroff equation (\ref{A1}) for $p(t,x,y)$, this directly implies 
\begin{align}\label{swaa}
\mathrm{tr}\left(\mathrm{e}^{-t H_p(w)}\right)\leq \int  p( t,x,x)\mathrm{e}^{-w(x)}\Id\mu(x) \quad\text{ for all $t> 0$,}
\end{align}
in case $w$ is bounded from below. Approximating $w$ with $w_n:=\max(-n,w)$, $n\in\IN$, and using the Feynman-Kac formula from Theorem \ref{awpo}, one sees that the formula (\ref{swaa}) remains valid for general $w$\rq{}s, too. In particular, for all $t>0$ one has
$$
\mathrm{tr}\left(\mathrm{e}^{-t H_p(w/t)}\right)\leq \int  p(t, x,x)\mathrm{e}^{-w(x)}\Id\mu(x),
$$
so that
$$
\varrho_1(t)\mathrm{tr}\left(\mathrm{e}^{-t H_p(w/t)}\right)\leq \int \varrho_1(t)p(t, x,x)\mathrm{e}^{-w(x)}\Id\mu(x).
$$
In view of the existence of
\begin{align}\label{asi}
\lim_{t \to 0+}  \varrho_1(t)p(t ,x,x) \quad\text{ for all $x\in X$,}
\end{align}
%so that the first summand in (\ref{auf}) tends to $\int_{K_n} \mathrm{e}^{-w}\Id\mu$ as $t\to 0+$. The second summand in (\ref{auf})  is
%$$
%\leq \int_{X\setminus K_n} (2\pit)^{m/2}\psi(t)\varrho(x) \mathrm{e}^{-w(x)}\Id\mu(x)\quad\text{ for all $0<t <1$,}
%$$
%which tends to $\int_{X\setminus K_n} (2\pi)^{m/2}\varrho(x) \mathrm{e}^{-w(x)}\Id\mu(x)$ as $t\to 0+$, by dominated convergence. 
and that for all $0<t<1$ and all $x\in X$ one has 
$$
\varrho_1(t)p(t ,x,x)  \mathrm{e}^{-w(x)}\lesssim  \mathrm{e}^{-w(x)}\varrho_2(x),
$$

which is a $\mu$-integrable function of $x\in X$, the inequality in (\ref{h1}) as well as 
$$
\int  \mathrm{e}^{-w(x)}\lim_{t\to 0+}p(t,x,x)\varrho_1(t) \Id\mu(x)<\infty
$$
follow from dominated convergence. \vspace{2mm}

\emph{Proof of the lower bound }
$$
\liminf_{t \to 0+}\varrho_1(t) \mathrm{tr}\big(\mathrm{e}^{-t  H_p(w/t )}\big)\geq  \int  \mathrm{e}^{-w(x)}\lim_{t\to 0}p(t,x,x)\varrho_1(t)  \Id\mu(x).
$$
Let $K_n$, $n\in \IN$, be an exhaustion of $X$ with open and relatively compact subsets. For each $n$ pick some $\delta_n\in (0,\infty]$ such that for all $0<\delta < \delta_n$ and all $x\in K_n$ the ball $B (x,\delta )$ is relatively compact. For example, one can take 
$$
\delta_n:=\inf_{z\in K_n}\min\big(1,\sup\{r>0: \text{$B (z,r )$ is relatively compact} \}\big)
$$
which is $>0$, as 
$$
X\ni z\longmapsto \min\big(1,\sup\{r>0: \text{$B (z,r )$ is relatively compact} \}\big)\in (0,\infty)
$$
is a Lipschitz function. Then for every $t >0$, $n\in\IN$, $0<\delta<\delta_n$, the Feynman-Kac formula from Theorem \ref{awpo} implies
\begin{align*}
&\varrho_1(t)\mathrm{tr}\left(\mathrm{e}^{-t  H_p(w/t )}\right)= \int  \varrho_1(t)p(t  ,x,x) \int   \mathrm{e}^{-\f{1}{t}\int^{t }_0w(\gamma(s))\Id s}  \Id\mathbb{P}^{x,x}_{t  }(\gamma) \ \Id\mu(x) \\
&\geq \int_{K_n} \varrho_1(t)p(t ,x,x) \int_{ \{\gamma: \ \gamma(s)\in \overline{B (x,\delta  )}\text{ for all $s\in [0,t )$}\}  } \mathrm{e}^{-\f{1}{t }\int^{t }_0w(\gamma(s))\Id s} \Id\mathbb{P}^{x,x}_{t  }(\gamma) \ \Id\mu(x) \\
&\geq \int_{K_n} \varrho_1(t)p(t ,x,x)  \mathbb{P}^{x,x}_{t  } \{\gamma: \ \gamma(s)\in \overline{B (x,\delta )}\text{ for all $s\in [0,t )$}\}  \mathrm{e}^{-w_{\delta}(x)}\Id\mu(x) ,
\end{align*}
where $w_{\delta}(x):=\max_{ \overline{B (x,\delta  )}}w$. Now, for all $x\in K_n$ one has 
$$
 \mathbb{P}^{x,x}_{t  } \{\gamma: \ \gamma(s)\in \overline{B (x,\delta )}\text{ for all $s\in [0,t )$}\} \to 1\quad\text{ as $t\to 0+$}
$$
by the principle of not feeling the boundary. Thus, by Fatou\rq{}s Lemma and by (\ref{asi}) we arrive at 
$$
\liminf_{t \to 0+} \varrho_1(t)\mathrm{tr}\left(\mathrm{e}^{-t  H_p(w/t )}\right)\geq \int_{K_n}  \mathrm{e}^{-w_{\delta}(x)}\liminf_{t \to 0+}p(t ,x,x)\varrho_1(t) \Id\mu(x) , 
$$
for all $n\in\IN$, $0<\delta <  \delta_n$.
Finally, the claim follows from taking first $\delta\to 0+$, and then $n\to \infty$, using dominated convergence twice.

\end{proof}

\section{Proof of Theorem \ref{Ricci}}\label{Rio}
  
\begin{proof} As $w^-$ is bounded by assumption, this function is in the
infinitesimal $Q_{g}:=Q_{p_g}$-class. It only remains to prove that
$$
\int_{X\setminus B_g(x_0,1)} 
\f{\mathrm{e}^{-w(x)}}{\mu_g(B_g(x,1))}\Id\mu_g(x)<\infty.
$$
To see the latter, set
$$
c_k:=\exp\Big(-\inf_{ x\in X, k-1<d_g(x,x_0)<k}w(x)\Big).
$$
We estimate as follows,
\begin{align}\label{adda}
\int_{X\setminus B_g(x_0,1)} 
\f{\mathrm{e}^{-w(x)}}{\mu_g(B_g(x,1))}\Id\mu_g(x)&=\sum_{k=2}^{\infty}\int_{\{x\in
X:  k-1 <d_g(x,x_0)<k\}} 
\f{\mathrm{e}^{-w(x)}}{\mu_g(B_g(x,1))}\Id\mu_g(x)\\\nn
&\leq \sum_{k=2}^{\infty}c_k\int_{\{x\in X: k-1<d_g(x,x_0)< k-1 \}} 
\f{1}{\mu_g(B_g(x,1))}\Id\mu_g(x).
\end{align}
The lower Ricci bound implies the following well-known \lq\lq{}doubling
property\rq\rq{} (cf. \cite{saloff}, p.420) for all $y\in X$, $k\in
\IN_{\geq 1}$,
$$
\mu_g(B_g(y,2k))\leq \mu_g(B_g(y,1)) (2k)^m \mathrm{e}^{L 2k}
$$
where $L :=\sqrt{(m-1)A }$, so that for all $x\in X$ with $d_g(x,x_0)<k$
we have
$$
\mu_g(B_g(x_0,k))\leq \mu_g(B_g(x,2k))\leq \mu_g(B_g(x,1)) (2k)^m
\mathrm{e}^{L 2k},
$$
and
$$
\f{1}{\mu_g(B_g(x,1))}\leq \f{1}{\mu_g(B_g(x_0,k))}(2k)^m \mathrm{e}^{L 2k}.
$$
As a consequence, (\ref{adda}) shows
\begin{align*}
\int_{X\setminus B_g(x_0,1)} 
\f{\mathrm{e}^{-w(x)}}{\mu_g(B_g(x,1))}\Id\mu_g(x)\leq
\sum_{k=2}^{\infty}c_k(2k)^m \mathrm{e}^{L 2k},
\end{align*}
which is finite by assumption.
\end{proof}


\begin{thebibliography}{99}

%\bibitem{aizi} Aizenman, M. \& Simon, B.: {\it Brownian motion and Harnack inequality for Schrï¿½dinger operators. } Comm. Purre. Appl. Math. 35 (2) (1982), 209--273. 

%\bibitem{avron}  Avron, J. \& Herbst, I. \& Simon, B.: {\it Schrödinger operators with magnetic fields. I. General interactions.} Duke Xath. J. 45 (1978), no. 4, 847-883.

%\bibitem{baer} Bär, C. \& Pfäffle, F.: \emph{Wiener measures on Riemannian manifolds and the Feynman-Kac formula.} Xat. Contemp. 40 (2011), 37--90.

\bibitem{baer2} Bär, C. \& Pfäffle, F.: \emph{Asymptotic heat kernel expansion in the semi-classical limit. } Comm. Math. Phys. 294 (2010), no. 3, 731--744.


%\bibitem{BGX} Bei, F. \& Güneysu, B. \& Xüller, J.: \emph{Scattering theory of the Hodge-Laplacian under a conformal perturbation}. To appear in J. Spectr.Theory (2015).


\bibitem{combes} Combes, J. M.; Schrader, R.; Seiler, R.: \emph{Classical bounds and limits for energy distributions of Hamilton operators in electromagnetic fields.} Ann Physics 111 (1978), no. 1, 1--18. 

%\bibitem{davies} Davies, E. B.:\emph{ Linear operators and their spectra.} Cambridge Studies in Advanced Xathematics, 106. Cambridge University Press, Cambridge, 2007.


%\bibitem{driver} Driver, B.K. \& Thalmaier, A.: \emph{Heat equation derivative formulas for vector bundles.} J. Funct. Anal. 183 (2001), no. 1, 42--108.


%\bibitem{enciso} Enciso, A.: {\it Coulomb Systems on Riemannian Xanifolds and Stability of Xatter.} Ann. Henri Poincare 12 (2011), 723--741.

\bibitem{palm} Fitzsimmons, P. \& Pitman, J. \& Yor, M.: \emph{Markovian bridges: construction, Palm interpretation, and splicing.} Seminar on Stochastic Processes, 1992 (Seattle, WA, 1992), 101--134, Progr. Probab., 33, Birkhäuser Boston, Boston, MA, 1993. 

\bibitem{fuku} Fukushima, M. \& Oshima, Y. \& Takeda, M.: \emph{Dirichlet forms and symmetric Markov processes.} De Gruyter Studies in Mathematics, 19. Walter de Gruyter \& Co., Berlin, 1994. 


%\bibitem{gri} Grigor'yan, A.: {\it Heat kernels on metric measure spaces with regular volume growth.} "Handbook of Geometric Analysis (Vol. II)" ed. L. Ji, P. Li, R. Schoen, L. Simon, Advanced Lectures in Xath. 13, International Press, 2010. 1--60.

%\bibitem{buch} Grigor'yan, A.: {\it Heat kernel and analysis on manifolds.} AXS/IP Studies in Advanced Xathematics, 47. American Xathematical Society, Providence, RI; International Press, Boston, XA, 2009.



%\bibitem{GXT} Güneysu, B. \& Xilatovic, O. \& Truc, F.: \emph{Generalized Schrödinger semigroups on infinite graphs.} Potential Anal. 41 (2014), no. 2, 517--541.

\bibitem{G45} Güneysu, B.: \emph{Semiclassical limits of quantum partition functions on infinite graphs.} J. Math. Phys. 56 (2015), no. 2, 022102, 13 pp. 


%\bibitem{G6} Güneysu, B.: {\it Nonrelativistic Hydrogen type stability problems on nonparabolic 3-manifolds}. Annales Henri Poincaré 13 (2012), 1557--1573.

%\bibitem{G1} G\"uneysu, B.: \emph{Kato's inequality and form boundedness of Kato potentials on arbitrary Riemannian manifolds.}  Proc. Amer. Math. Soc. 142 (2014), no. 4, 1289--1300.


 

\bibitem{gun} Güneysu, B.: \emph{Heat Kernels in the Context of Kato Potentials on Arbitrary Manifolds}. Potential Analysis (2016). DOI 10.1007/s11118-016-9574-x.

%\bibitem{semi} Güneysu, B.: \emph{On the semimartingale property of Brownian bridges on complete manifolds. } Preprint (2016).

\bibitem{G2} Güneysu, B.: {\it On generalized Schrödinger semigroups.} J. Funct. Anal. 262 (2012), 4639--4674.

\bibitem{post} Güneysu, Batu \& Post, O.: \emph{Path integrals and the essential self-adjointness of differential operators on noncompact manifolds.} Math. Z. 275 (2013), no. 1-2, 331--348.

\bibitem{habil} Güneysu, B.: \emph{Covariant Schrödinger semigroups on noncompact manifolds.} Habilitationsschrift, Humboldt-Universität 2016.

\bibitem{GKS-13}  Güneysu, B. \& Keller, M. \& Schmidt, M.: \emph{A Feynman-Kac-Itô formula for magnetic Schrödinger operators on graphs.} Probab. Theory Related Fields 165 (2016), no. 1-2, 365--399.

%\bibitem{gmt} Güneysu, B. \& Xilatovic, O. \& Truc, F.: \emph{Generalized Schrödinger semigroups on infinite graphs.} Potential Anal. 41 (2014), no. 2, 517--541. 


\bibitem{helfer} Helffer, B. \& Robert, D.: \emph{Calcul fonctionnel par la transformation de Mellin et opérateurs admissibles.} J. Funct. Anal. 53 (1983), no. 3, 246--268.

\bibitem{helfer2} Helffer, B. \& Sjöstrand, J.: \emph{On diamagnetism and de Haas-van Alphen effect.} Ann. Inst. H. Poincaré Phys. Théor. 52 (1990), no. 4, 303--375. 

%\bibitem{HPW} Hempel, R. \& Post, O. \& Weder, R.: \emph{On open scattering channels for manifolds with ends}, J. Funct. Anal. 266 (2014), no. 9, 5526--5583.

%\bibitem{hess} Hess, H. \& Schrader, R. \& Uhlenbrock, D. A.: {\it Domination of semigroups and generalization of Kato's inequality.} Duke Xath. J. 44 (1977), no. 4, 893--904.

\bibitem{nu} Hiai, F.: \emph{ Log-majorizations and norm inequalities for exponential operators.} Banach Cent. Publ. 38(1), 119--181 (1997).


\bibitem{Hsu} Hsu, E.: {\it Stochastic Analysis on Manifolds.} AMS, 2002.

 \bibitem{Hsu2} Hsu, E: \emph{On the principle of not feeling the boundary for diffusion processes.} J. London Math. Soc. (2) 51 (1995), no. 2, 373--382. 

%\bibitem{Jo} Johnson, G.W. \& Lapidus, X. L.: {\it The Feynman integral and Feynman's operational calculus.} The Clarendon Press, Oxford University Press, 2000.




%\bibitem{katze3} Kato, T.: {\it Perturbation theory for linear operators.} (Reprint of the 1980 edition.) Springer-Verlag, Berlin, 1995.

%\bibitem{katze2} Kato, T.: {\it Perturbation theory for linear operators.} Die Grundlehren der mathematischen Wissenschaften, Band 132 Springer-Verlag New York, Inc., New York 1966.

%\bibitem{kato2} Kato, T.: {\it Schrï¿½dinger operators with singular potentials.} Israel J. Xath. 13 (1972).
 
\bibitem{kl} Keller, M. \& Lenz, D.: \emph{Dirichlet forms and stochastic completeness of graphs and subgraphs.} J. Reine Angew. Xath. 666 (2012), 189--223.

%\bibitem{Kenyon-11} Kenyon, R.:
%\emph{Spanning forests and the vector bundle Laplacian.} Ann. Probab. 39 (2011), 1983--2017.

\bibitem{kt} Kuwae, K. \& Takahashi, X.: {\it Kato class measures of symmetric Xarkov processes under heat kernel estimates.}  J. Funct. Anal. 250 (2007), no. 1, 86--113.

%\bibitem{ku2} Kuwae, K. \& Takahashi, X.: {\it Kato class functions of Xarkov processes under ultracontractivity.} Potential theory in Xatsue, 193--202, Adv. Stud. Pure Xath., 44, Xath. Soc. Japan, Tokyo, 2006.
%
%\bibitem{lenz} Lenz, D. \& Keller, X. \& Vogt, H. \& Wojciechowski, R.: {\it Note on basic features of large time behaviour of heat kernels.} To appear in J. reine angew. Xath.
%
%\bibitem{lenzi} Lenz, D. \& Stollmann, P. \& Wingert, D.: \emph{Compactness of Schrodinger semigroups}. Xath. Nachrichten 283 (2010), 94--103. 

%\bibitem{li} Li, P. \& Tam, L.-F.: {\it Symmetric Green's functions on complete manifolds.} Amer. J. Xath.  109  (1987),  no. 6, 1129?1154.

%\bibitem{lieb} Lieb, E.H. \& Loss, X.: {\it Analysis.} Second edition. Graduate Studies in Xathematics, 14. American Xathematical Society, Providence, RI, 2001.
%
%\bibitem{Kenyon-11} Kenyon, R.:
%\emph{Spanning forests and the vector bundle Laplacian.} Ann. Probab. 39 (2011), 1983--2017.

%\bibitem{ma} Xa, Z.X. \& Röckner, X.: {\it Introduction to the theory of (nonsymmetric) Dirichlet forms.} Universitext. Springer-Verlag, Berlin, 1992.

%\bibitem{villa} Xaggi, F. \& Villani, C.: {\it Balls have the worst best Sobolev inequalities.} J. Geom. Anal. 15, 1 (2005), 83--121. 

%\bibitem{X1} Xilatovic, O.: {\it "Localized'' self-adjointness of Schrï¿½dinger type operators on Riemannian manifolds.} J. Xath. Anal. Appl.  283  (2003),  no. 1, 304--318.

%\bibitem{Petersen} Petersen, P.: {\it Riemannian geometry.} Second edition. Graduate Texts in Xathematics, 171. Springer, New York, 2006.
%
%\bibitem{pitt} Pitt, L.D. {\it A compactness condition for linear operators of function spaces. } J. Operator Theory 1 (1979), no. 1, 49--54. 


%\bibitem{Re} Reed, X. \& Simon, B.: {\it Xethods of modern mathematical physics. I. Functional analysis.} Academic Press, Inc., 1972.

%\bibitem{Re2} Reed, X. \& Simon, B.: {\it Xethods of modern mathematical physics. II. Fourier analysis, self-adjointness.} Academic Press, Inc., 1975.

\bibitem{Re3} Reed, M. \& Simon, B.: \emph{Methods of modern mathematical physics. III. Scattering theory.} Academic Press [Harcourt Brace Jovanovich, Publishers], New York-London, 1979.

%\bibitem{Re4} Reed, X. \& Simon, B.: {\it Xethods of modern mathematical physics. IV. Analysis of operators.} Academic Press, Inc., 1978.


\bibitem{saloff} Saloff-Coste, L.: \emph{Uniformly elliptic operators on Riemannian manifolds.} J. Differential Geom. 36 (1992), no. 2, 417--450. 

%\bibitem{schrader} Schrader, R. \&  Taylor, M.E.:  \emph{Small $\hbar$ asymptotics for quantum partition functions associated to particles in external Yang-Mills potentials.} Commun. Math. Phys. 92(4), 555--594 (1994).


\bibitem{simon} Simon, B.: \emph{Functional Integration and Quantum Physics.} Second edition. AMS Chelsea Publishing, Providence, RI: Amer. Math. Soc., 2005.

%
%\bibitem{Si} Simon, B.: {\it Trace ideals.}.

\bibitem{peter} Stollmann, P. \& Voigt, J.: {\it Perturbation of Dirichlet forms by measures.} Potential Anal. 5 (1996), no. 2, 109--138.


%
%\bibitem{Su} T. Sunada, A discrete analogue of periodic magnetic Schr\"odinger operators,
%Geometry of the spectrum,  Contemp. Xath. \textbf{173}, Amer. Xath. Soc.,
%Providence, RI, 1994, 283--299.

\bibitem{sz} Sznitman, A.-S.: \emph{Brownian motion, obstacles and random media.} Springer Monographs in Mathematics. Springer-Verlag, Berlin, 1998.

%\bibitem{teschl} Teschl, G.: {\it Xathematical methods in quantum mechanics. With applications to Schrödinger operators.} Graduate Studies in Xathematics, 99. American Xathematical Society, Providence, RI, 2009.

%\bibitem{weid} Weidmann, J.: {\it Linear operators in Hilbert spaces.} Graduate Texts in Xathematics, 68. Springer-Verlag, New York-Berlin, 1980.

\end{thebibliography}
\end{document}